\documentclass{amsart}
\usepackage{amsthm}
\usepackage{enumerate}
\usepackage{amsmath,amssymb}
\usepackage{tikz-cd}

\usepackage{xr-hyper} 

\usepackage{hyperref}

\newtheorem{theorem}{Theorem}[section]
\newtheorem{lemma}[theorem]{Lemma}
\newtheorem{proposition}[theorem]{Proposition}
\newtheorem{corollary}[theorem]{Corollary}
\theoremstyle{definition}
\newtheorem{remark}[theorem]{Remark}
\numberwithin{equation}{section}
\allowdisplaybreaks[1]

\DeclareMathOperator{\End}{End}
\DeclareMathOperator{\Hom}{Hom}

\DeclareMathOperator{\Res}{Res}
\DeclareMathOperator{\ord}{ord}
\DeclareMathOperator{\aff}{aff}

\DeclareMathOperator{\GL}{GL}
\DeclareMathOperator{\SL}{SL}
\DeclareMathOperator{\prim}{prim}
\DeclareMathOperator{\Rep}{Rep}
\newcommand{\IEC}{\mathfrak{I}}
\DeclareMathOperator{\sgn}{sgn}
\DeclareMathOperator{\Lie}{Lie}
\DeclareMathOperator{\Gal}{Gal}
\DeclareMathOperator{\norm}{norm}
\DeclareMathOperator{\Tr}{Tr}

\renewcommand{\restriction}{|}

\newcommand{\tG}{{\widetilde G }}
\newcommand{\tM}{{\widetilde M }}

\newcommand{\tT}{{\widetilde T }}
\newcommand{\Krel}{{\mathcal{K}\textup{-rel}}}
\newcommand{\flength}{{\ell_{\Krel}}}
\newcommand{\isoarrow}{\stackrel{\sim}{\longrightarrow}}
\newcommand{\muT}{\mu^{\cT}}

\newcommand{\Waff}{W(\rho_M)_{\mathrm{aff}}}
\newcommand{\Wzero}{\Omega(\rho_{M})}
\newcommand{\Wzeroz}{\Omega(\rho_{M^0})}
\newcommand{\plength}{\ell_{\prim}}
\newcommand{\phm}{\phantom{-}}

\newcommand{\bC}{\mathbb{C}}
\newcommand{\bZ}{\mathbb{Z}}

\newcommand{\cB}{\mathcal{B}}
\newcommand{\cH}{\mathcal{H}}
\newcommand{\cO}{\mathcal{O}}
\newcommand{\cT}{\mathcal{T}}

\newcommand{\ff}{\mathfrak{f}}
\newcommand{\fp}{\mathfrak{p}}
\newcommand{\fs}{{\mathfrak{s}}}
\newcommand{\fsz}{{\mathfrak{s}_0}}

\title[A depth-zero principal-series block...]%
{A depth-zero principal-series block whose Hecke algebra has a non-trivial two-cocycle}
\author{Jeffrey D. Adler}
\address{Department of Mathematics and Statistics\\
        American University\\
        4400 Massachusetts Ave NW\\
        Washington, DC 20016-8050, USA}
\email{jadler@american.edu} 
\author{Jessica Fintzen}
\address{Universität Bonn\\
	Mathematisches Institut\\
	Endenicher Allee 60\\
	53115 Bonn\\
	Germany}
\email{fintzen@math.uni-bonn.de}
\author{Kazuma Ohara}
\address{Universität Bonn\\
	Mathematisches Institut\\
	Endenicher Allee 60\\
	53115 Bonn\\
	Germany}
\email{kohara@math.uni-bonn.de}

\begin{document}

\externaldocument[HAI-][nocite]{Adler--Fintzen--Mishra--Ohara_Structure_of_Hecke_algebras_arising_from_types}[https://arxiv.org/pdf/2408.07801]
\externaldocument[HAIKY-][nocite]{Adler--Fintzen--Mishra--Ohara_Reduction_to_depth_zero_for_tame_p-adic_groups_via_Hecke_algebra_isomorphisms}[https://arxiv.org/pdf/2408.07805]

\begin{abstract}
Recently the authors have shown that every Hecke algebra
associated to a type constructed by Kim and Yu is isomorphic to a Hecke algebra
for a depth-zero type. 
An example in the literature has been suggested
as a counterexample to this result.
We show that the example is not a counterexample, and exhibit
some of its interesting properties, e.g., we show that a principal series, depth-zero type can have a Hecke algebra with non-trivial 2-cocyle, a phenomenon that many did not expect could occur.
\end{abstract}

{
	\renewcommand{\thefootnote}{}  
	\footnotetext{%
		The first-named author was also partially supported by
		an AMS-Simons Research Enhancement Grant for PUI Faculty,
		and
		the American University College of Arts and Sciences Faculty Research Fund.
		The second- and third-named authors were partially supported by 
		the European Research Council (ERC) under the European Union's Horizon 2020 research
		and innovation programme (grant agreement n° 950326).}          
}

\maketitle

\section{Introduction}
Let $G$ denote a connected reductive group over a non-archimedean local field $F$.
The category $\Rep(G(F))$ of smooth, complex representations of $G(F)$
is a direct product of full subcategories
called ``Bernstein blocks'':
\[
\Rep(G(F))=\prod_{\fs\in\IEC(G)}\Rep^\fs(G(F)).
\]
Each of the blocks $\Rep^\fs(G(F))$
is equivalent to the category of unital right modules
over an algebra $\cH^\fs$.
Suppose that the category $\Rep^\fs(G(F))$ has
an associated \emph{type}, as defined by Bushnell and Kutzko
\cite{BK-types}, i.e., a compact open subgroup $K$ of $G(F)$
and an irreducible smooth representation $\rho$ of $K$ such that
a representation $\pi \in \Rep(G(F))$ belongs to $\Rep^\fs(G(F))$
if and only if every irreducible subquotient of $\pi$
contains $\rho$ upon restriction to $K$.
Then we can replace the algebra $\cH^\fs$
by the Hecke algebra $\cH(G(F),(K, \rho))$
of all compactly supported,
$\End_{\bC}(\rho)$-valued
functions on $G(F)$
that transform on the left and right
according to $\rho$.
That is, $\Rep^\fs(G(F))$ is equivalent to the category
of modules over $\cH(G(F),(K, \rho))$.

One of the present authors \cite{Fi-exhaustion} has shown that,
provided that $G$ splits over a tamely ramified extension of $F$ and the residual characteristic $p$ of $F$ is not too small,
the construction of Kim and Yu \cite{Kim-Yu}
provides types for every Bernstein block for $G(F)$.
Thus, under this mild tameness assumption, one can in principle understand
the category $\Rep(G(F))$ by understanding the structures
of all of the Hecke algebras that arise from the types constructed by Kim and Yu.

In the ``depth-zero'' case, the compact group $K$
contains a parahoric subgroup of $G(F)$, and the representation
$\rho$ is trivial on the pro-$p$ radical of this parahoric.
In the special case where $K$ is a parahoric subgroup,
Morris \cite[Theorem 7.12]{Morris} has described the structures of these Hecke algebras.
More generally, the authors
\cite[Theorem \ref{HAI-theoremstructureofheckefordepthzero}]{HAI}
have described the structures of all depth-zero types.

In the set-up of Kim and Yu (\cite{Kim-Yu}), a type $(K,\rho)$ for $G$ of positive depth
is constructed from a depth-zero type $(K^0,\rho^0)$ for a subgroup $G^0$ of $G$,
together with some additional data.
The authors have recently shown
\cite[Theorem \ref{HAIKY-theoremstructureofheckeforKimYu}]{HAIKY}
that the associated Hecke algebras
$\cH(G(F), (K, \rho))$ and $\cH(G^0(F),(K^0, \rho^0))$ are isomorphic, after twisting the construction by Kim and Yu by a quadratic character arising from \cite{FKS},
thus describing 
the structures of all Hecke algebras arising from Kim--Yu types.

In outline,
from the pair $(K,\rho)$, one constructs
a group $W^\heartsuit$,
and
a normal, affine reflection subgroup $W_{\aff}$ of $W^\heartsuit$.
Choosing a set $S$ of generating reflections for $W_{\aff}$,
one constructs a parameter function $q\colon S \longrightarrow \bC^\times$,
thus obtaining an abstract Hecke algebra $\cH(W_{\aff},q)$.
The choice of $S$ gives rise to a complement $\Omega$
to $W_{\aff}$ in $W^\heartsuit$.
A choice of a family $\cT$ of intertwining operators
gives rise to a $2$-cocycle
$\muT \colon \Omega \times \Omega \longrightarrow \bC^\times$.
One then obtains an isomorphism of $\bC$-algebras
\[
\cH(G(F),(K, \rho)) \isoarrow \bC[\Omega,\muT] \ltimes \cH(W_{\aff},q).
\]
That is, $\cH(G(F),(K, \rho))$ is isomorphic to a semi-direct product
of our abstract Hecke algebra and the $\muT$-twisted group algebra of $\Omega$,
where the structure of multiplication between these two factors is controlled
by the conjugation action of $\Omega$ on $W_{\aff}$.

In relation to the above discussion, Roche \cite[\S4]{Roche:parabolic} and
Goldberg--Roche \cite[\S11.8]{GoldbergRoche-Hecke}
each illustrate some unusual phenomena by presenting
an example, that they attribute to Kutzko,
of a Hecke algebra $\cH := \cH(G(F), (K, \rho))$ for a particular block of $G = \SL_{8}$.
In this note,
we discuss this example, determine an attached depth-zero pair $(K^0, \rho^0)$ and describe the depth-zero algebra
$\cH^0 := \cH(G^0(F),(K^0, \rho^0))$
that corresponds to it via 
\cite[Theorem \ref{HAIKY-theoremstructureofheckeforKimYu}]{HAIKY} explicitly, as well as the closely related Hecke  algebra $\cH^{0,\circ} :=\cH(G^0(F),(K^{0, \circ}, \rho^0))$, where we replace $K^0$ by the parahoric subgroup $K^{0,\circ}$ contained in it.
We have several aims in doing so.
\begin{enumerate}[(a)]
\item
First, 
a
remark in \cite{GoldbergRoche-Hecke} that
$\cH$ cannot be isomorphic
to any of the intertwining algebras constructed by
Morris \cite{Morris} let several mathematicians believe that $\cH$ would provide a counterexample to 
\cite[Theorem \ref{HAIKY-theoremstructureofheckeforKimYu}]{HAIKY}.
Thus, we want to assure readers that this is not the case.
We address this matter, and other apparent conficts with earlier results,
in \S\ref{sec:comparison}.
\item
Second, $\cH^0$ 
provides an example of a depth-zero Hecke algebra
where the associated affine reflection group is trivial,
the group $\Wzero$ is nonabelian and infinite,
and the cocycle $\muT$ is non-trivial.
In particular, $\cH^0$ is an example of a Hecke algebra attached to a depth-zero, principal-series block of a quasi-split group that requires a non-trivial 2-cocycle, something that was long believed not to exist. 
We believe that this example might be useful
for researchers in the future.
\end{enumerate}

\subsection*{Acknowledgement}
The authors thank Maarten Solleveld for comments on an earlier version of this paper.

\subsection*{AI declaration}
The authors have not used AI in this paper.  (The first version of this paper was posted on arxiv.org on 8 October 2025.) 

\subsection*{Notation}

For a connected reductive group $G$ and a reductive subgroup $M$ of $G$, let $N_{G}(M)$, resp., $Z_{G}(M)$, denote the normalizer, resp., centralizer, of $M$ in $G$.

For a finite field extension $E/F$ and $A$ a linear algebraic group or a Lie algebra thereof defined over $E$, we write $\Res_{E/F}(A)$ for the Weil restriction of $A$ to $F$.

For a linear algebraic group $G$, we denote by $\Lie(G)$ the Lie algebra of $G$ and by $\Lie^{*}(G)$ the dual of $\Lie(G)$.
We also write $\Lie^{*}(G)^{G}$ for the subscheme of $\Lie^{*}(G)$ fixed by the coadjoint action of $G$ on $\Lie^{*}(G)$.
For a morphism $f \colon G \to H$ of algebraic groups, let 
\[
\Lie(f) \colon \Lie(G) \to \Lie(H)
\]
denote the morphism between their Lie algebras induced by $f$.

Suppose that $G$ is a connected reductive group defined over a non-archimedean local field $F$.
We denote by $\mathcal{B}(G, F)$ the enlarged Bruhat–Tits building of $G$.
For $x \in \mathcal{B}(G, F)$, let $G(F)_{x}$ denote the stabilizer of $x$ in $G(F)$.
For $r \in \mathbb{R}$ with $r \ge 0$, we also let $G(F)_{x, r}$, resp., $G(F)_{x, r+}$, be the Moy--Prasad filtration subgroup of $G(F)$ of depth $r$, resp., $r+$, associated to $x$ (see \cite{MR1253198, MR1371680}).
We use the analogous notation for the Lie algebra $\Lie(G)$ and its dual $\Lie^{*}(G)$, where $r$ is allowed to be any element of $\mathbb{R}$.

For a compact, open subgroup $K$ of $G(F)$ and an irreducible smooth representation $\rho$ of $K$, we denote by $\cH(G(F), (K, \rho))$ the Hecke algebra attached to $(K, \rho)$.
We refer to \cite[Section~2.2]{HAI} for the precise definition of $\cH(G(F), (K, \rho))$.

Suppose that $K$ is a subgroup of a group $H$ and $h \in H$. We denote $hKh^{-1}$ by $^hK$. If $\rho$ is a representation of $K$, we write $^h\!\rho$ for the representation $x\mapsto \rho(h^{-1}xh)$ of $^hK$.

\section{The example}
In this section we introduce the example studied in this paper that we learned about from
Roche \cite[\S4]{Roche:parabolic} and
Goldberg--Roche \cite[\S11.8]{GoldbergRoche-Hecke},
who attribute it to Kutzko.
We present it, a type for the group $\SL_8$, in the language of Kim and Yu's construction of types.
Doing so then allows us to describe the associated depth-zero type for a smaller group $G^0$,
seeing directly that it fits into our set-up.

Let $F$ denote a non-archimedean local field with residue field $\ff$
of characteristic $p$ (assumed odd) and order $q$.
We fix a uniformizer $\varpi_F$ of $F$
and a square root $\sqrt{-1}$ of $-1$ in $\bC^\times$.
For any finite field extension $E/F$,
we denote by $\cO_E$ the ring of integers in $E$,
by $\fp_E$ the prime ideal in $\cO_E$, by $\Tr_{E/F} \colon E \to F$ the trace map, and by
$N_{E/F} \colon E^\times \to F^\times$
 the norm map.
Let $\ord$ denote the discrete valuation on $F^{\times}$ with the value group $\mathbb{Z}$.
 For any finite extension $E$ of $F$ , we also write $\ord$ for the unique extension of this valuation to $E^{\times}$.
We denote by $\ord^{\norm}_E \colon E^{\times} \twoheadrightarrow \bZ$ the normalized valuation on $E^{\times}$.

Let $\zeta$ be a primitive $(q - 1)$-st root of unity in $F$.
Assume that $4$ divides $q - 1$.
It follows that there exists a unique character $\eta \colon F^{\times} \to \mathbb{C}^{\times}$ that is trivial on $\varpi_{F}$ and $1 + \fp_{F}$ and satisfies $\eta(\zeta) = \sqrt{-1}$.
Let $E_{2}$ be the splitting field of the polynomial $X^{2} + \varpi_{F}$,
and $E_{4}$ the splitting field of the polynomial $X^{4} + \zeta \varpi_{F}$.
(Note that the fields that we denote by $E_{2}$ and $E_{4}$ here are denoted by $E_{1}$ and $E_{2}$, respectively, in \cite{Roche:parabolic}.)
Let $\varpi_{E_{2}}$, resp., $\varpi_{E_{4}}$,
denote a uniformizer of $E_{2}$, resp., $E_{4}$,
such that $\varpi_{E_{2}}^{2} = - \varpi_{F}$, resp., $\varpi_{E_{4}}^{4} = - \zeta \varpi_{F}$.
We fix a generator $\sigma_{2}$ of the Galois group $\Gal(E_2/F)$ and a generator $\sigma_{4}$ of the Galois group $\Gal(E_{4}/F)$.

We define the following reductive groups over $F$:
\begin{align*}
\tG  = \tG^2 &= \GL_8, \\
\tG^1 &= \Res_{E_2/F} (\GL_2) \times \GL_4, \\
\tG^0 &= \Res_{E_2/F} (\GL_2) \times \Res_{E_4/F} (\GL_1), \\
\tT = \tM^0 
&= \Res_{E_2/F} \left(
\GL_1 \times \GL_1
\right) \times \Res_{E_4/F} (\GL_1).
\end{align*}
We identify $\GL_{1} \times \GL_{1}$ with the diagonal torus of $\GL_{2}$ by the map
\(
(t_{1}, t_{2}) \mapsto 
(\begin{smallmatrix}
t_{1} & 0 \\
0 & t_{2}
\end{smallmatrix}),
\)
thus obtaining an
embedding $\tM^0 \hookrightarrow \tG^0$.
Fix isomorphisms
$F^{\oplus 2} \cong E_2$
and
$F^{\oplus 4} \cong E_4$
of vector spaces over $F$,
thus determining isomorphisms
$$
F^{\oplus 8}
\quad \cong \quad
E_2^{\oplus 2} \oplus F^{\oplus 4}
\quad \cong \quad
E_2^{\oplus 2} \oplus E_4.
$$
These choices determine embeddings of $F$-groups
\(
\tG^0
\hookrightarrow \tG^1
\hookrightarrow \tG^2.
\)
Let us identify each group above with its images under these maps,
so that they are all contained within $\tG = \GL_8$.

The maximal split subtorus $A_{\tT}$ of $\tT$
is isomorphic to $\GL_1 \times \GL_1 \times \GL_1$.
Note that $\tM^0 = Z_{\tG^0}(A_{\tT})$.
For $i=1,2$, let
$\tM^i = Z_{\tG^i}(A_{\tT})$,
and write $\tM = \tM^2$.

Let $G = \SL_8$.
For $X \in \{G^i, M^i, M, T \, | \, i = 0, 1, 2\}$,
we let 
$X = \widetilde X \cap G$.
We thus obtain a twisted Levi sequence
$(G^0 \subset G^1 \subset G^2 = \SL_8)$,
and a Levi subgroup $M^0 \subset G^0$.
We denote by $\Phi(X, T)$ the absolute root system of $X$ with respect to the maximal torus $T$.
For $\alpha \in \Phi(X, T)$, we denote by $\alpha^{\vee}$ the corresponding (absolute) coroot.

We let $\widetilde{K}^{0}$ be the Iwahori subgroup of $\tG^0(F)$ given by $\widetilde{K}^{0} = I_{2} \times I_{4}$, where $I_{2}$ denotes the Iwahori subgroup of $\GL_{2}(E_{2}) = \left(
\Res_{E_{2}/F}(\GL_{2})
\right)(F)$ defined by
\[
I_{2} =
\biggl\{
\begin{pmatrix}
a & b \\
c & d
\end{pmatrix} \in \GL_{2}(E_{2})
\biggr. \,\biggl|\,
a, d \in \cO_{E_{2}}^{\times}, b \in \cO_{E_{2}}, c \in \fp_{E_{2}}
\biggr\},
\]
and $I_{4}$ denotes the Iwahori subgroup of $\left(
\Res_{E_{4}/F}(\GL_{1})
\right)(F) = E_4^{\times}$, i.e., $I_{4} = \cO_{E_{4}}^{\times}$. 

We choose $x_{0} \in \cB(M^{0}, F)$ and fix a commutative diagram $\{ \iota \}$
\[
\begin{tikzcd}
	\cB(G^0, F)
	\arrow[r, hook]
	\arrow[dr, pos=.5, phantom, "\circlearrowleft"]
	&
	\cB(G^1, F)
	\arrow[r, hook]
	\arrow[dr, pos=.5, phantom, "\circlearrowleft"]
	&
	\cB(G^2, F)
	\\
	\cB(M^0, F)
	\arrow[u, hook]
	\arrow[r, hook]
	&
	\cB(M^1, F)
	\arrow[u, hook]
	\arrow[r, hook]
	&
	\cB(M^2, F)
	\arrow[u, hook]
\end{tikzcd} 
\]
of admissible embeddings of buildings that is $(0, \tfrac18, \tfrac14)$-generic relative to $x_{0}$ in the sense of \cite[3.5~Definition]{Kim-Yu} such that $G^0(F)_{x_0}=\widetilde K^0 \cap G^0(F)$.
Here and from now on, we identify a point in $\cB(M^{0}, F)$ with its images via the embeddings $\{ \iota \}$. 
Then we have
\[
G^0(F)_{x_{0}, 0} = \left\{
(g_2, g_4) \in (I_{2} \times I_{4}) \cap G^0(F) \mid (\det(g_2) \bmod \mathfrak{p}_{E_{2}}) \cdot (g_{4} \bmod \mathfrak{p}_{E_{4}})^{2} = 1
\right\},
\]
where we regard $(\det(g_2) \bmod \mathfrak{p}_{E_{2}})$ and $(g_{4} \bmod \mathfrak{p}_{E_{4}})$ as elements of $\mathfrak{f}^{\times}$.
Let $K^0$ be either $G^{0}(F)_{x_0}$ or $G^0(F)_{x_0, 0}$.
We also define $\widetilde{K}_{M^0} = \widetilde{K}^0 \cap \widetilde{M}^0(F)$ and $K_{M^0} = K^{0} \cap M^0(F)$.
Thus, we have $\widetilde{K}_{M^0} = \cO_{E_{2}}^{\times} \times \cO_{E_{2}}^{\times} \times \cO_{E_{4}}^{\times}$ and
$K_{M^0} = M^0(F)_{x_0}$ or $M^0(F)_{x_0,0}$
according as $K^0 = G^0(F)_{x_0}$ or $G^0(F)_{x_0,0}$.
More precisely, 
\begin{equation}
\label{descriptionofKM0}
K_{M^0} = 
\begin{cases}
\left\{
(x, y, z) \in \cO_{E_{2}}^{\times} \times \cO_{E_{2}}^{\times} \times \cO_{E_{4}}^{\times}
\mid
N_{E_{2}/ F}(xy) N_{E_{4}/ F}(z) = 1 
\right\} \\
\hspace{20em}\text{if $K^0 = G^0(F)_{x_0}$,}
\\[1ex]
\left\{
(x, y, z) \in \cO_{E_{2}}^{\times} \times \cO_{E_{2}}^{\times} \times \cO_{E_{4}}^{\times}
\;\middle|\;
\begin{aligned}
& N_{E_{2}/F}(xy)\, N_{E_{4}/F}(z) = 1, \\
& (xy \bmod \mathfrak{p}_{E_{2}})\cdot (z \bmod \mathfrak{p}_{E_4})^{2} = 1
\end{aligned}
\right\}
\\
\hspace{20em}\text{if $K^0 = G^0(F)_{x_0,0}$.}
\end{cases}
\end{equation}
We observe that
\begin{equation}
\label{connectedparahoricvsdisconenctedparahoric}
K^{0} = K_{M^0} \cdot G^0(F)_{x_0, 0}.
\end{equation}
Since the embedding $\iota \colon \cB(M^0, F) \hookrightarrow \cB(G^0, F)$ is $0$-generic relative to $x_{0}$, the inclusion $M^0(F)_{x_0,0} \subset G^0(F)_{x_0,0}$ induces an isomorphism 
\[
M^0(F)_{x_0, 0}/M^0(F)_{x_{0}, 0+} \isoarrow G^0(F)_{x_0, 0}/G^0(F)_{x_0, 0+}.
\]
Combining this with \eqref{connectedparahoricvsdisconenctedparahoric}, we also have
\begin{equation}
\label{isomofreductivequotients}
K_{M^0}/M^0(F)_{x_{0}, 0+} \isoarrow K^0/G^0(F)_{x_0, 0+}.
\end{equation}

We define the character $\widetilde\rho_{M^0}$ of $\widetilde{K}_{M^0}$ by $\widetilde\rho_{M^0} = 1 \boxtimes \left(
\eta \circ N_{E_{2} / F}
\right) \boxtimes 1$, and write $\rho_{M^0} = \widetilde\rho_{M^0} \restriction_{K_{M^0}}$.
We define the character $\rho^0$ of $K^0$ as the composition of the surjection $K^0 \twoheadrightarrow K^0/G^0(F)_{x_{0}, 0+}$, the inverse of the isomorphism in \eqref{isomofreductivequotients} and the character $\rho_{M^0}$.
More precisely, $\rho^0$ is the restriction to $K^0$ of the character $\eta_{2} \boxtimes 1$ of the group $\widetilde K^0$, where $\eta_{2}$ denotes the character of $I_{2}$ defined by
\[
\eta_{2} \left(
\begin{pmatrix}
a & b \\
c & d
\end{pmatrix}
\right) = \left(
\eta \circ N_{E_{2} / F}
\right)(d).
\]

Let $E = E_{2}$ or $E_{4}$.
We fix an additive character $\Psi \colon F \rightarrow \bC^{\times}$ that is trivial on $\fp_{F}$ and non-trivial on $\cO_{F}$.
We define a character $\phi_{E}$ of $1 + \fp_{E}$ by $\phi_{E}(1 + x) = \Psi(\Tr_{E/F}(\varpi_{E}^{-1} x))$ for $x \in \fp_{E}$.
We fix an extension of $\phi_{E}$ to $E^{\times}$ and use the same notation $\phi_{E}$ for it.
We also define the character $\phi_{\GL_2(E_2)}$ of $\GL_{2}(E_{2})$ by $\phi_{\GL_2(E_2)}(g) = \phi_{E_{2}}(\det(g))$.
We define the character $\widetilde{\phi}_{0}$ of $\tG^0(F) = \GL_{2}(E_{2}) \times E_{4}^{\times}$ by $\widetilde{\phi}_{0} = 1 \boxtimes \phi_{E_{4}}$, and define the character $\widetilde{\phi}_{1}$ of $\tG^1(F) = \GL_2(E_2) \times \GL_4(F)$ by $\widetilde{\phi}_{1} = \phi_{\GL_2(E_2)} \boxtimes 1$.
We write $\phi_{0} = \widetilde{\phi}_{0} \restriction_{G^0(F)}$ and $\phi_{1} = \widetilde{\phi}_{1} \restriction_{G^1(F)}$.
\begin{lemma}
The character $\phi_{0}$ is $(G^{1}, G^{0})$-generic of depth $\tfrac14$ relative to the point $x_{0}$, and the character $\phi_{1}$ is $(G^{2}, G^{1})$-generic of depth $\tfrac12$ relative to the point $x_{0}$ in the sense of \cite[Definition~3.5.2]{Fintzen-IHES}.
\end{lemma}
\begin{proof}
By construction, $\phi_0$ is trivial on $G^0(F)_{x_0,(1/4)+}$, and  $\phi_1$ is trivial on $G^1(F)_{x_0,(1/2)+}$.
We define $\widetilde X_{0}^{*} \in \Lie^{*}(\tG^{0})^{\tG^0}(F)$ and $\widetilde X_{1}^{*} \in \Lie^{*}(\tG^{1})^{\tG^1}(F)$ as follows.
Let $E \in \{E_2, E_4\}$.
We use the same notation 
\[
\Tr_{E/F} \colon \Res_{E/F}(\Lie(\GL_1)) \to \Lie(\GL_1)
\]
for the usual trace morphism whose map on $F$-valued points is the trace map.
Let
\[
m(\varpi_{E}^{-1}) \colon \Lie\left(
\Res_{E/F}(\GL_{1})
\right) \to \Lie\left(
\Res_{E/F}(\GL_{1})
\right)
\]
denote the morphism induced by multiplication by $\varpi_{E}^{-1} \in E = \Lie\left(
\Res_{E/F}(\GL_{1})
\right)(F)$.
We define $\widetilde X_{0}^{*} \in \Lie^{*}(\tG^{0})^{\tG^0}(F)$ as the composition of the projection map
\[
\Lie(\tG^{0}) \rightarrow \Lie\left(
\Res_{E_{4}/F}(\GL_{1})
\right)
\]
and
\[
\Tr_{E_{4}/F}
 \circ m(\varpi_{E_{4}}^{-1}) \colon \Lie\left(
\Res_{E_{4}/F}(\GL_{1})
\right) \to \Lie(\GL_{1}).
\]
To define $\widetilde X_{1}^{*}$, we let 
\[
\Res_{E_{2}/F}(\det) \colon \Res_{E_{2}/F}(\GL_{2}) \to \Res_{E_2/F}(\GL_{1})
\]
be the morphism of algebraic groups induced by the usual determinant map $\det \colon \GL_{2} \to \GL_{1}$.
Now, we define $\widetilde X_{1}^{*} \in \Lie^{*}(\tG^{1})^{\tG^1}(F)$ as the composition of the projection map
\[
\Lie(\tG^{1}) \rightarrow \Lie\left(
\Res_{E_{2}/F}(\GL_{2})
\right)
\]
and
\[
\Tr_{E_{2}/F}
 \circ m(\varpi_{E_{2}}^{-1})  \circ \Lie\left(
\Res_{E_{2}/F}(\det)
\right) \colon \Lie\left(
\Res_{E_{2}/F}(\GL_{2})
\right) \to \Lie(\GL_{1}).
\]
We define $X_{0}^{*} \in \Lie^{*}(G^{0})^{G^0}(F)$ and $X_{1}^{*} \in \Lie^{*}(G^{1})^{G^1}(F)$ as the restrictions of $\widetilde X_{0}^{*}$ and $\widetilde X_{1}^{*}$ to $\Lie(G^0)$ and $\Lie(G^1)$, respectively. 
Then the restriction of $\phi_{0}$ to 
\[
G^{0}(F)_{x_{0}, 1/4}/G^0(F)_{x_0,(1/4)+} \simeq \Lie(G^0)(F)_{x_{0}, 1/4}/\Lie(G^0)(F)_{x_{0}, (1/4)+}
\]
 is given by $\Psi \circ X_{0}^{*}$, and the restriction of $\phi_{1}$ to 
\[
G^{1}(F)_{x_{0}, 1/2}/G^1(F)_{x_0,(1/2)+} \simeq \Lie(G^1)(F)_{x_{0}, 1/2}/\Lie(G^1)(F)_{x_{0}, (1/2)+}
\]
 is given by $\Psi \circ X_{1}^{*}$.

We will prove that $X_{0}^{*}$ is $(G^1, G^0)$-generic of depth $- 1/4$, and $X_{1}^{*}$ is $(G^2, G^1)$-generic of depth $-1/2$ in the sense of \cite[Definition~3.5.2]{Fintzen-IHES}.
First, we will prove that $X_{0}^{*}$ satisfies {\bf (GE0)} and {\bf (GE1)} in \cite[Definition~3.5.2]{Fintzen-IHES}.
Let $\alpha \in \Phi(G^1, T) \smallsetminus \Phi(G^{0}, T)$.
Then we have
\[
X^{*}_{0}(\Lie( \alpha^{\vee})(1)) = \sigma^{i}_{4}(\varpi_{E_{4}}^{-1}) - \sigma^{j}_{4}(\varpi_{E_{4}}^{-1})
\]
for some $i, j \in \{0, 1, 2, 3\}$ with $i \neq j$.
Since $E_{4} = F[\varpi_{E_{4}}^{-1}]$, we obtain from \cite[Proposition~5.9]{2020arXiv200106259M} that
\[
\ord\left(
\sigma^{i}_{4}(\varpi_{E_{4}}^{-1}) - \sigma^{j}_{4}(\varpi_{E_{4}}^{-1})
\right) = \ord(\varpi_{E_{4}}^{-1}) = -1/4.
\]
Thus, the element $X_{0}^{*}$ satisfies {\bf (GE1)} in \cite[Definition~3.5.2]{Fintzen-IHES}.
Moreover, since $(0, \varpi_{E_{4}}) \in \Lie(G^0)_{x_{0}, 1/4}$ (where we view $\varpi_{E_{4}}$ in $\Lie(\Res_{E_4/F}(\GL_1))(F)$ by identifying the latter with $E_4$ and note that $(0, \varpi_{E_{4}}) \in \Lie(G^0)(F)$ as $\varpi_{E_4}$ has trace zero) and 
\[
\ord\left(
X^{*}_{0}(0, \varpi_{E_{4}})
\right) = \ord(4) = 0,
\]
we have $X^{*}_{0} \not \in \Lie^*(G^0)_{x_{0}, (-1/4)+}$.
Since it can be checked from the definition that $X_{0}^{*} \in \Lie^*(G^0)_{x_0, -1/4}$, the element $X_{0}^{*}$ also satisfies {\bf (GE0)} in \cite[Definition~3.5.2]{Fintzen-IHES}.

Next, we will prove that $X_{1}^{*}$ satisfies {\bf (GE0)} and {\bf (GE1)} in \cite[Definition~3.5.2]{Fintzen-IHES}.
Let $\alpha \in \Phi(G^2, T) \smallsetminus \Phi(G^{1}, T)$.
Then, using that $\sigma_2(\varpi_{E_2})=-\varpi_{E_2}$, we obtain
\[
X^{*}_{1}(\Lie(\alpha^{\vee})(1)) \in \{\pm\varpi_{E_{2}}^{-1}, \pm 2\varpi_{E_{2}}^{-1}\} .
\]
Hence
\[
\ord\left(
X^{*}_{1}(\Lie(\alpha^{\vee})(1))
\right) = \ord(\varpi_{E_{2}}^{-1}) = -1/2 ,
\]
and $X_{1}^{*}$ satisfies {\bf (GE1)} in \cite[Definition~3.5.2]{Fintzen-IHES}.
Consider the element 
$\left(\begin{pmatrix}
	\varpi_{E_2} & 0 \\
	0 & 0 
\end{pmatrix}, 0\right) \in \Lie^*(G^1)_{x_0, -1/2}$, then
\[\ord\left( X_1^* \left(\begin{pmatrix}
	\varpi_{E_2} & 0 \\
	0 & 0 
\end{pmatrix}, 0\right) \right)=\ord(2)=0 ,
\]
thus  $X_1^*  \notin \Lie^*(G^1)_{x_0, -1/2+}$.
Moreover, $X_1^*  \in \Lie^*(G^1)_{x_0, -1/2}$, hence $X_{1}^{*}$ satisfies {\bf (GE0)} in \cite[Definition~3.5.2]{Fintzen-IHES}.

Since the only possible torsion prime for the dual root datum of $G^1$ and $G^2$ is 2, and since $p \neq 2$, by \cite[Lemma~8.1]{Yu} condition {\bf (GE2)} is also satisfied.
\end{proof}

As a consequence of the above lemma, the datum
\[
\Sigma = \bigl(
(G^0 \subset G^1 \subset G^2, M^0), (\tfrac14, \tfrac12, \tfrac12), (x_{0}, \{\iota\}), (K_{M^0}, \rho_{M^0}), (\phi_{0}, \phi_{1}, 1)
\bigr)
\]
satisfies the properties of \cite[(7.2)]{Kim-Yu}, in other words, it is a $G$-datum as in \cite[Definition~4.1.1]{HAIKY}.
Applying the construction of Kim and Yu in \cite[7.4]{Kim-Yu} to $\Sigma$, we obtain a compact, open subgroup $K$ of $G(F)$ and an irreducible smooth representation $\rho$ of $K$.

\begin{remark}
In \cite{HAIKY}, we twist the construction of Kim and Yu by a quadratic character $\epsilon_{x_{0}}^{\overrightarrow{G}}$ of $K^{0}/G^0(F)_{x_{0}, 0+} \simeq K_{M^0}/M^0(F)_{x_{0}, 0+}$ introduced in \cite{FKS}, see \cite[\S4.1]{HAIKY} for details.
In our case, we can compute $\epsilon_{x_{0}}^{\overrightarrow{G}}$ using \cite[Definition~3.1, Theorem~3.4]{FKS} as follows:
\[
\epsilon_{x_{0}}^{\overrightarrow{G}}\left(
(x, y, z) \bmod M^0(F)_{x_{0}, 0+}
\right) = \sgn_{\ff}(xy^{-1} \bmod \mathfrak{p}_{E_{2}}) = \sgn_{\ff}(xy \bmod \mathfrak{p}_{E_{2}})
\]
for $(x, y, z) \in K_{M^0}$, where $\sgn_{\ff}$ denotes the unique non-trivial quadratic character of $\ff^{\times} = \mathcal{O}^{\times}_{E_{2}}/(1+\mathfrak{p}_{E_{2}})$.
Since $-1 \in (\ff^{\times})^{2}$, the conditions in \eqref{descriptionofKM0} imply that 
\[
xy \bmod \mathfrak{p}_{E_{2}} = \pm (z \bmod \mathfrak{p}_{E_{4}})^{-2} \in (\ff^\times)^2.
\]
Thus, we obtain that $\epsilon_{x_{0}}^{\overrightarrow{G}}$ is trivial, and the twisted and non-twisted constructions agree in our case.
\end{remark}

According to
\cite[Theorem \ref{HAIKY-theoremstructureofheckeforKimYu}]{HAIKY},
we have an isomorphism of $\bC$-algebras
\[
\cH(G^0(F), (G^0(F)_{x_0}, \rho^0)) \isoarrow \cH(G(F), (K, \rho)).
\]
In the following section, we determine explicitly the structure of the Hecke algebras $\cH(G^0(F),(G^0(F)_{x_0}, \rho^0))$ and $\cH(G^0(F),(G^0(F)_{x_0,0}, \rho^0))$.

\section{Structure of the depth-zero Hecke algebra}
\label{sec:depth0}

In this section, we will study the Hecke algebra $\cH(G^0(F), (K^0, \rho^0))$ associated with the depth-zero type $(K^0, \rho^0)$.
We define the subgroup $N(\rho_{M^0})$
of the $F$-points of the normalizer $N_{G^0}(M^0)$ of $M^0$ in $G^0$ by
\[
N(\rho_{M^0}) :=
\left\{
n \in N_{G^0}(M^0)(F) \mid {^n\!K_{M^0}} = K_{M^0}, \, {^n\!\rho_{M^0}} = \rho_{M^0}
\right\}
\]
and write $W(\rho_{M^0}) := N(\rho_{M^0})/K_{M^0}$. 
We write 
\[
I_{G^0(F)}(\rho^0)
:=
\{
g \in G^0(F)
\mid
\Hom_{K^0 \cap ^g\! K^0}({^g\! \rho^0}, \rho^0) \neq \{0\}
\}.
\]
Then from 
\cite[Proposition~\ref{HAI-propproofofaxiombijectionofdoublecoset}
and
Corollary~\ref{HAI-corollarybijectiondoublecosetandweylgroup}]{HAI},
we have $N(\rho_{M^0}) \subset I_{G^0(F)}(\rho^0)$ and the inclusion induces a bijection
\[
W(\rho_{M^0}) \simeq K^0 \backslash I_{G^0(F)}(\rho^0) / K^0.
\]
In order to describe the groups $N(\rho_{M^0})$ and $W(\rho_{M^0})$ below, we define the element $\widetilde{s}$ of the group
\[
G^0(F) = \left(
\GL_{2}(E_{2}) \times \GL_{1}(E_{4})
\right) \cap \SL_{8}(F) \supset \SL_{2}(E_{2}) \times \SL_{1}(E_{4})
\]
by
$
\widetilde{s} = \left(
(\begin{smallmatrix}
\phm 0 & 1 \\
-1 & 0
\end{smallmatrix}), 1
\right) 
$.
Then we have $N_{G^0}(M^0)(F) = \{1, \widetilde{s}\} \ltimes M^0(F)$.
We note that the element $\widetilde{s}$ normalizes the groups $\widetilde{K}_{M^0}$ and $K_{M^0}$.
\begin{lemma}
\label{lemmasnormalizesthecharacterrhoM}
The element $\widetilde{s}$ normalizes the character $\rho_{M^0}$.
\end{lemma}
\begin{proof}
We have
\begin{align*}
^{\widetilde{s}}\!\widetilde\rho_{M^0} &= {^{\widetilde{s}}\!\Bigl(
1 \boxtimes (
\eta \circ N_{E_{2} / F}
) \boxtimes 1
\Bigr)} \\
&= 
(\eta \circ N_{E_{2} / F}) \boxtimes 1 \boxtimes 1 \\
&= \Bigl(
1 \boxtimes (\eta^{-1} \circ N_{E_{2} / F}) \boxtimes (\eta^{-1} \circ N_{E_{4} / F})
\Bigr) \otimes 
(\eta \circ \det),
\end{align*}
where $\det$ denotes the restriction of the determinant map $\GL_{8}(F) \to F^{\times}$ to the group $\widetilde{K}_{M^0}$.
Since the group $K_{M^0}$ is contained in the group $\SL_{8}(F)$, we have
\begin{equation}
\label{calculationofsrhoM}
^{\widetilde{s}}\!\rho_{M^0} = \Bigl(
1 \boxtimes (\eta^{-1} \circ N_{E_{2} / F}) \boxtimes (\eta^{-1} \circ N_{E_{4} / F})
\Bigr)\restriction_{K_{M^0}}.
\end{equation}
We will prove that
\begin{equation}
\label{etasquareE1trivial}
\eta^{2} \circ N_{E_{2} / F} \restriction_{\cO_{E_{2}}^{\times}} = 1
\end{equation}
and
\begin{equation}
\label{etaE2trivial}
\eta \circ N_{E_{4}/F} \restriction_{\cO_{E_{4}}^{\times}}  = 1.
\end{equation}
First, we will prove Equation~\eqref{etasquareE1trivial}.
The definition of $E_{2}$ implies that we have 
\[
N_{E_{2} / F}(\cO_{E_{2}}^{\times}) = (1 + \fp_{F}) \langle \zeta^{2} \rangle.
\]
Hence, Equation~\eqref{etasquareE1trivial} follows from the definition of $\eta$.
Similarly, we can prove Equation~\eqref{etaE2trivial} by using the definition of $\eta$ and the equation
\[
N_{E_{4} / F}(\cO_{E_{4}}^{\times}) = (1 + \fp_{F}) \langle \zeta^{4} \rangle.
\]

Combining equation~\eqref{calculationofsrhoM} with Equations~\eqref{etasquareE1trivial} and \eqref{etaE2trivial}, we obtain that
\begin{align*}
^{\widetilde{s}}\!\rho_{M^0} &= \Bigl(
1 \boxtimes (\eta^{-1} \circ N_{E_{2} / F}) \boxtimes (\eta^{-1} \circ N_{E_{4} / F})
\Bigr)\restriction_{K_{M^0}} \\
&= 
 \Bigl(
1 \boxtimes (\eta \circ N_{E_{2} / F}) \boxtimes 1
\Bigr)\restriction_{K_{M^0}} \\
&= \widetilde\rho_{M^0}\restriction_{K_{M^0}} = \rho_{M^0}.
\qedhere
\end{align*}
\end{proof}
\begin{proposition}
\label{propositionstructureofNrhoMSL8example}
We have
\[
N(\rho_{M^0}) = N_{G^0}(M^0)(F) = \{1, \widetilde{s}\} \ltimes M^0(F).
\]
\end{proposition}
\begin{proof}
The claim $N(\rho_{M^0}) \subset N_{G^0}(M^0)(F)$ follows from the definition of $N(\rho_{M^0})$.
We will prove the reverse inclusion.
According to Lemma~\ref{lemmasnormalizesthecharacterrhoM}, we have $\widetilde{s} \in N(\rho_{M^0})$.
Moreover, since $M^0$ is a torus, the conjugate action of $M^0(F)$ on $K_{M^0}$ is trivial.
Thus, we conclude that the group $M^0(F)$ normalizes the character $\rho_{M^0}$.
\end{proof}

To describe the structure of the group $W(\rho_{M^0})$,
we define the elements $\widetilde{s}' \in N(\rho_{M^0})$ and $\widetilde{z} \in M^0(F)$ by
$
\widetilde{s}' = \left(
\left(\begin{smallmatrix}
0 & \varpi_{E_{2}}^{-1} \\
-\varpi_{E_{2}} & 0
\end{smallmatrix}\right), 1
\right) 
$ 
and $\widetilde{z} = \left(
\zeta \varpi_{E_{2}}, \varpi_{E_{2}}, \varpi_{E_{4}}^{-2}
\right)$.
Let $s$, $s'$, and $z$ be the images of $\widetilde{s}$, $\widetilde{s}'$, and $\widetilde{z}$ in $W(\rho_{M^0})$, respectively.
When $K^0=G^0(F)_{x_0,0}$, we also set $\widetilde \epsilon_{M^0} := (-1, 1, 1) \in M^0(F)_{x_0} \smallsetminus M^0(F)_{x_0,0}$ and let $\epsilon_{M^0}$ denote the image of $\widetilde \epsilon_{M^0}$ in $W(\rho_{M^0})$.

\begin{corollary}
\label{cor:W-as-lattice-with-action}
We have
\(
W(\rho_{M^0}) = \{1, s\} \ltimes ( M^0(F) / K_{M^0} ).
\)
\end{corollary}

\begin{proof}
This is immediate from Proposition \ref{propositionstructureofNrhoMSL8example}.
\end{proof}

\begin{proposition}
\label{propositionstructureofWrhoMSL8example}
We have 
\[
W(\rho_{M^0}) =
\begin{cases}
 \langle
s, s'
\rangle \times \langle z \rangle \phantom{,\epsilon_{M^0}}
\simeq W_{\mathrm{aff}} (\widetilde A_1) \times \bZ & (K^0 = G^0(F)_{x_0}), \\
 \langle
s, s'
\rangle \times \langle z, \epsilon_{M^0} \rangle
\simeq W_{\mathrm{aff}} (\widetilde A_1) \times \bZ \times \bZ/2\bZ& (K^0 = G^0(F)_{x_0,0}),
\end{cases}
\]
where $W_{\mathrm{aff}}(\widetilde A_1)$ denotes the affine Weyl group of the affine root system of type $\widetilde A_1$, i.e., the affine Weyl group with two simple reflections and no relation between them.
\end{proposition}
\begin{proof}
First, we consider the case where $K^0 = G^0(F)_{x_0}$.
We define the isomorphism
\[
H_{M^0} \colon \tM^0(F) / \widetilde{K}_{M^0}  \rightarrow \mathbb{Z}^{3}
\]
by
\[
(x, y, z) \widetilde{K}_{M^0}  \mapsto (\ord^{\norm}_{E_{2}}(x), \ord^{\norm}_{E_{2}}(y), \ord^{\norm}_{E_{4}}(z)).
\]
The definition of $M^0$ implies that an element $(n_{1}, n_{2}, n_{3}) \in \mathbb{Z}$ is contained in the image of the subgroup
\[
M^0(F) / K_{M^0} \simeq M^0(F) \cdot \widetilde{K}_{M^0}  / \widetilde{K}_{M^0} 
\]
of $\tM^0(F) / \widetilde{K}_{M^0}$
if and only if
\begin{equation}
\label{conditionaboutnormmaps}
N_{E_{2}/F}(\varpi_{E_{2}}^{n_{1} + n_{2}}) \cdot N_{E_{4} / F}(\varpi_{E_{4}}^{n_{3}}) \in N_{E_{2}/F}(\cO_{E_{2}}^{\times}) \cdot N_{E_{4}/F}(\cO_{E_{4}}^{\times}).
\end{equation}
The definition of $E_{2}$ and $E_{4}$ implies that \eqref{conditionaboutnormmaps} is equivalent to the condition that
\[
\varpi_{F}^{n_{1} + n_{2} + n_{3}} \cdot \zeta^{n_{3}} \in (1 + \fp_{F}) \langle \zeta^{2} \rangle.
\]
Thus, we conclude that 
\begin{align*}
H_{M^0} \left(
M^0(F)/K_{M^0}
\right) &= \left\{
(n_{1}, n_{2}, n_{3}) \in \mathbb{Z} \mid n_{1} + n_{2} + n_{3} = 0 \:\text{and} \:  2 \mid n_{3}
\right\} \\
&= \langle 
(1, 1, -2), (1, -1, 0)
\rangle.
\end{align*}
Since
$
H_{M^0}(z) = (1, 1, -2)
$,
$H_{M^0}(s s') = (1, -1, 0)$, and $H_{M^0}$ is an isomorphism, we have that 
\[
M^0(F)/K_{M^0} = \langle 
ss', z
\rangle. 
\]
From Corollary \ref{cor:W-as-lattice-with-action},
we thus obtain that
\[
W(\rho_{M^0}) = \{1, s\} \ltimes \left(
M^0(F) / K_{M^0}
\right) = \langle
s, s', z
\rangle.
\]
One can check that the subgroup $\langle
s, s'
\rangle$ of $W(\rho_{M^0})$ is isomorphic to the affine Weyl group of type $\widetilde{A_{1}}$, and that the element $z$ has infinite order, commutes with the elements $s$ and $s'$, and no non-trivial power of $z$ is contained in the span of $s$ and $s'$.

Next, we consider the case where $K^0 = G^0(F)_{x_0,0}$.
In this case, we have
\[
M^0(F)/K_{M^0} = M^0(F)/M^0(F)_{x_0,0} = \left(
M^0(F)/M^0(F)_{x_0}
\right) \times \langle \epsilon_{M^0} \rangle.
\]
Noting that $\epsilon_{M^0}$ commutes with $s$ and $s'$, the claim follows from the first case.
\end{proof}

\begin{lemma}
\label{lemmacanttakenznscommute}
Let $n_{s}, n_{z} \in N(\rho_{M^0})$
denote lifts of $s$ and $z$.
Then, we have $[n_{s}, n_{z}] \in K_{M^0} \smallsetminus \ker \rho_{M^0}$.
In particular, we have $n_{s} n_{z} \neq n_{z} n_{s}$.
\end{lemma}
\begin{proof}
We write $n_{s} = \widetilde{s} k$ and $n_{z} = \widetilde{z} k'$
for some $k, k' \in K_{M^0}$.
Then, we have
\begin{align*}
[n_{s}, n_{z}] &= [\widetilde{s} k, \widetilde{z} k'] \\
&= [\widetilde{s} k, \widetilde{z}] \cdot ^{\widetilde{z}}\![\widetilde{s} k, k'] \\
&= {^{\widetilde{s}}[k, \widetilde{z}]} \cdot [\widetilde{s}, \widetilde{z}] \cdot ^{\widetilde{z}}\![\widetilde{s} k, k'].
\end{align*}
Since $\widetilde{s}, \widetilde{z} \in N(\rho_{M^0})$ normalize $\rho_{M^0}$, we have
\[
^{\widetilde{s}}[k, \widetilde{z}], ^{\widetilde{z}}\![\widetilde{s} k, k'] \in \ker \rho_{M^0}.
\]
On the other hand, we have 
\[
[\widetilde{s}, \widetilde{z}] = (\zeta^{-1}, \zeta, 1) \in K_{M^0} \smallsetminus \ker \rho_{M^0}
\]
since
\[
\rho_{M^0}((\zeta^{-1}, \zeta, 1)) = \left(
\eta \circ N_{E_{2}/F}
\right)(\zeta) = \eta(\zeta^{2}) = -1.
\]
Thus, we conclude that $[n_{s}, n_{z}] \in K_{M^0} \smallsetminus \ker \rho_{M^0}$.
\end{proof}

\begin{corollary}
\label{corollaryrhoMdoesnotextend}
The character $\rho_{M^0}$ does not extend to the group $N(\rho_{M^0})$.
\end{corollary}
\begin{proof}
Suppose that $\rho_{M^0}$ extends to a character $\rho_{M^0}^{\dagger}$of $N(\rho_{M^0})$.
Then, since $\widetilde{s}, \widetilde{z} \in N(\rho_{M^0})$,
we have $[\widetilde{s}, \widetilde{z}] \in \ker \rho_{M^0}^{\dagger}$,
which contradicts Lemma~\ref{lemmacanttakenznscommute}.
\end{proof}

Our decomposition of $W(\rho_{M^0})$ given in
Proposition \ref{propositionstructureofWrhoMSL8example}
gives rise to a length function on this group, the standard length function on extended affine Weyl groups. More precisely, we start with the length function $\plength$ on $\langle
s, s'
\rangle$ with respect to the generators $\{s, s'\}$ of $\langle
s, s'
\rangle$.
We extend $\plength$ to $W(\rho_{M^0})$ by
\[
\begin{cases}
\plength(w z^{n}) := \plength(w) & \text{ when } K^0 = G^0(F)_{x_0}, \\
\plength(w z^{n} \epsilon_{M^0}^{t}) := \plength(w) & \text{ when } K^0 = G^0(F)_{x_0,0}
\end{cases}
\]
for $w \in \langle
s, s'
\rangle$, $n \in \mathbb{Z}$, and $t \in \{0, 1\}$.

\begin{remark}
\label{rmk:no-good-coset-reps}
Suppose that $K^0 = G^0(F)_{x_0,0}$.
According to \cite[Proposition~5.2]{Morris}, we can take a lift $n_{w} \in N(\rho_{M^0})$ for each $w \in W(\rho_{M^0})$ such that if $\plength(w_{1} w_{2}) = \plength(w_{1}) + \plength(w_{2})$, then we have $n_{w_{1} w_{2}} = n_{w_{1}} n_{w_{2}}$.
However, this statement is false in general,
and Lemma~\ref{lemmacanttakenznscommute} provides a counterexample. 
The failure of \cite[Proposition~5.2]{Morris} does not affect Morris's main result \cite[Theorem 7.12]{Morris} as his proof can easily be adapted to circumvent the use of such good coset representatives. Alternatively, the recent proof of the more general result \cite[Section~5]{HAI} also does not rely on a choice of representatives.
On the other hand, \cite[Remark 7.12(a)]{Morris}, which states that the the $2$-cocycle $\muT$ is trivial if the representation $\rho^{0}$ is a character, does depend on such representatives, and the example covered in this paper shows that \cite[Remark 7.12(a)]{Morris} is not true in general (see Corollary~\ref{corollaryanexampleofheckealgebrawithnon-trivialtwococycle} below).
\end{remark}

Although 
Proposition \ref{propositionstructureofWrhoMSL8example}
decomposes $W(\rho_{M^0})$ into a product
of an affine Weyl group $\langle s,s' \rangle$
and a complement ($\langle z \rangle$ and $\langle z, \epsilon_{M^0} \rangle$ for $K^0 = G^0(F)_{x_0}$ and $K^0 = G^0(F)_{x_0,0}$, respectively)
the decomposition of $W(\rho_{M^0})$ provided in 
\cite{HAI} (and also in \cite[7.3]{Morris}) is different, and comes from a different length function, where more elements have length zero,
which is denoted by $\flength$ and defined in \cite[Definition \ref{HAI-definitionlength}]{HAI}.
The subgroup $\Wzeroz$ of $W(\rho_{M^0})$ is defined by
\[
	\Wzeroz := \Bigl\{
	w \in W(\rho_{M^0}) \Bigr. \,\Bigl | \, \flength(w) = 0
	\Bigr\}.
\]

\begin{proposition}
\label{propositionWheartisOmegainSL8example}
We have $W(\rho_{M^0}) = \Omega(\rho_{M^0})$.
\end{proposition}

That is, all elements of $W(\rho_{M^0})$ have length zero
with respect to $\flength$.

\begin{proof}
For each $w \in W(\rho_{M^0})$, we fix a non-zero element $\varphi_{w} \in \cH(G^0(F), (K^0, \rho^0))$
with support in ${K}^0 w {K}^0$.
To prove the proposition, it suffices to show that for any $w_{1}, w_{2} \in W(\rho_{M^0})$, we have $\varphi_{w_{1}} * \varphi_{w_{2}} \in \mathbb{C} \cdot \varphi_{w_{1} w_{2}}$.
According to Proposition~\ref{propositionstructureofWrhoMSL8example}, we have
$
W(\rho_{M^0}) = \langle
s, s'
\rangle \times \langle z \rangle
$ or $
W(\rho_{M^0}) = \langle
s, s'
\rangle \times \langle z, \epsilon_{M^0} \rangle
$,
and we can check easily that if $w_{1}, w_{2} \in W(\rho_{M^0})$ satisfy $\plength(w_{1} w_{2}) = \plength(w_{1}) + \plength(w_{2})$, then we have 
\[
\widetilde{K}^0 w_{1} \widetilde{K}^0 w_{2} \widetilde{K}^0 = \widetilde{K}^0 w_{1} w_{2} \widetilde{K}^0.
\]
Since we have $\widetilde{K}^0 = \widetilde{K}_{M^0} \cdot K^0$, and the group $N(\rho_{M^0})$ normalizes the group $\widetilde{K}_{M^0}$, we also obtain that 
\begin{align*}
K^0 w_{1} K^0 w_{2} K^0 & \subseteq \widetilde{K}^0 w_{1} \widetilde{K}^0 w_{2} \widetilde{K}^0 \cap \SL_{8}(F) \\
&= \widetilde{K}^0 w_{1} w_{2} \widetilde{K}^0 \cap \SL_{8}(F) \\
&= \widetilde{K}^0 w_{1} w_{2} \widetilde{K}_{M^0} \cdot K^0 \cap \SL_{8}(F) \\
&= \widetilde{K}^0 \cdot  \widetilde{K}_{M^0}  w_{1} w_{2} K^0 \cap \SL_{8}(F) \\
&= \widetilde{K}^0 w_{1} w_{2} K^0 \cap \SL_{8}(F) \\
&= \left(
\widetilde{K}^0 \cap \SL_{8}(F)
\right) w_{1} w_{2} K^0 \\
&= K^0 w_{1} w_{2} K^0.
\end{align*}
In particular, in this case, we obtain that $\varphi_{w_{1}} * \varphi_{w_{2}} \in \mathbb{C} \cdot \varphi_{w_{1} w_{2}}$.
Thus, to prove the proposition, it now suffices to show that $\varphi_{s} * \varphi_{s} \in \mathbb{C} \cdot \varphi_{1}$ and $\varphi_{s'} * \varphi_{s'} \in \mathbb{C} \cdot \varphi_{1}$.
Similar calculations as above imply that $K^0 s K^0 s K^0 = K^0 \cup K^0 s K^0$ and $K^0 s' K^0 s' K^0 = K^0 \cup K^0 s' K^0$.
Hence, we obtain that
\[
\varphi_{s} * \varphi_{s} \in \mathbb{C} \cdot \varphi_{s} \oplus \mathbb{C} \cdot \varphi_{1}
\qquad
\text{and}
\qquad
\varphi_{s'} * \varphi_{s'} \in \mathbb{C} \cdot \varphi_{s'} \oplus \mathbb{C} \cdot \varphi_{1}.
\]
Thus, it suffices to prove that $\left(
\varphi_{s} * \varphi_{s}
\right)(\widetilde{s}) = \left(
\varphi_{s'} * \varphi_{s'}
\right)(\widetilde{s}') = 0$.
We take a set of representatives for 
\(
K^0 / \left(
K^0 \cap {^{\widetilde{s}}\!K^0}
\right)
\)
as
\(
\left\{
u(x) \mid x \in \cO_{E_{2}} / \fp_{E_{2}}
\right\}
\),
where
$
u(x) = \left(
\begin{pmatrix}
1 & x \\
0 & 1
\end{pmatrix}, 1
\right)
$.
Then, we can calculate the convolution product $\left(
\varphi_{s} * \varphi_{s}
\right)(\widetilde{s})$ as
\begin{align*}
\left(
\varphi_{s} * \varphi_{s}
\right)(\widetilde{s}) 
&= \sum_{h \in K^0 \widetilde{s} K^0 / K^0} \varphi_{s}(h) \cdot \varphi_{s}(h^{-1}\widetilde{s}) \\
&= \sum_{k \in K^0 / \left(
K^0 \cap {^{\widetilde{s}}\!K^0}
\right)
} \varphi_{s}(k\widetilde{s}) \cdot \varphi_{s}(\widetilde{s}^{-1} k^{-1} s) \\
&= \sum_{x \in \cO_{E_{2}} / \fp_{E_{2}}}  \varphi_{s}(u(x) \widetilde{s}) \cdot \varphi_{s}(\widetilde{s}^{-1} u(-x) \widetilde{s}).
\end{align*}
For $x \in \cO_{E_{2}}$, we have
\begin{align*}
\widetilde{s}^{-1} u(-x) \widetilde{s} 
=
\left(
\begin{pmatrix}
1 & 0 \\
x & 1
\end{pmatrix}, 1
\right).
\end{align*}
Hence, $\widetilde{s}^{-1} u(-0) \widetilde{s} \not \in K^0 \widetilde{s} K^0$ and for $x \in \cO_{E_{2}}^{\times}$, we have
\begin{align*}
\widetilde{s}^{-1} u(-x) \widetilde{s} &= 
\left(
\begin{pmatrix}
-x^{-1} & -1 \\
0 & -x 
\end{pmatrix}
\cdot
\widetilde{s}
\cdot
\begin{pmatrix}
1 & x^{-1} \\
0 & 1
\end{pmatrix}, 1
\right).
\end{align*}
Hence, the definition of $\rho^0$ implies that
\begin{align*}
\left(
\varphi_{s} * \varphi_{s}
\right)(\widetilde{s}) &= \sum_{x \in \cO_{E_{2}} / \fp_{E_{2}}}  \varphi_{s}(u(x) \widetilde{s}) \cdot \varphi_{s}(\widetilde{s}^{-1} u(-x) \widetilde{s}) \\
&= \varphi_{s}(\widetilde{s})^{2} \sum_{x \in \cO_{E_{2}}^{\times} /\left(1+ \fp_{E_{2}}\right)} \left(
\eta \circ N_{E_{2}/F}
\right)(-x) \\
&= \varphi_{s}(\widetilde{s})^{2} \sum_{x \in \cO_{F}^{\times} /\left(1+ \fp_{F}\right)} \eta(x^{2}) = 0,
\end{align*}
where the last equality follows from the fact that the restriction of the character $\eta^{2}$ to $\cO_{F}^{\times}$ is non-trivial.
Similarly, we can prove that $\left(
\varphi_{s'} * \varphi_{s'}
\right)(\widetilde{s}') = 0$.
\end{proof}

We fix a family
$
\cT =
\left\{
T_{n} \in \Hom_{K_{M^0}}\left(
^n\!\rho_{M^0}, \rho_{M^0}
\right)
\right\}_{n \in N(\rho_{M^0})}
$
as in \cite[Choice \ref{HAI-choice:tw}]{HAI}
and define the $2$-cocycle 
\[
\muT \colon W(\rho_{M^0}) \times W(\rho_{M^0}) \rightarrow \bC^{\times}
\]
as in \cite[Notation \ref{HAI-notationofthetwococycle}]{HAI}.

\begin{corollary}
\label{corollaryanexampleofheckealgebrawithnon-trivialtwococycle}
We have an isomorphism
\[
\cH(G^0(F), (K^0, \rho^0)) \simeq \mathbb{C}[W(\rho_{M^0}), \muT],
\]
and the $2$-cocycle $\muT$ is non-trivial.
\end{corollary}
\begin{proof}
The corollary follows from
\cite[Theorem \ref{HAI-thm:isomorphismtodepthzero}]{HAI},
Corollary~\ref{corollaryrhoMdoesnotextend},
and
Proposition~\ref{propositionWheartisOmegainSL8example}.
\end{proof}

\section{Comparison with earlier results}
\label{sec:comparison}
There are several ways in which our calculations at least \emph{appear} to conflict with earlier results.
Because such apparent conflicts could cause concern for readers,
we address three of them here.

We have provided a presentation of $\cH(G^0(F), (K^0, \rho^0))$, and thus of $\cH(G(F), (K,\rho))$,
in two cases: where $K^0 = G^0(F)_{x_0}$, and where $K^0 = G^0(F)_{x_0,0}$.
The first case gives a Hecke algebra associated to a single Bernstein block,
and our presentation is in the style of Morris's presentations of Hecke algebras \cite{Morris}.
The second case is literally a special case of the situation
treated by Morris,
and our presentation is the same as his.
Let us restrict our attention to this case for now.

First, because our representation $\rho^0$ is a character,
Remark 7.12(a) of \cite{Morris} would imply that our cocycle
$\muT$ is trivial.
Since it clearly is not, our algebra $\cH(G(F),(K, \rho))$
appears not to match up with any algebra constructed by Morris.
However, Morris's observation depends on the existence of good coset representatives as in \cite[Proposition~5.2]{Morris}, which do not exist in general, and in particular in this example, as explained in Remark~\ref{rmk:no-good-coset-reps}.

Second,
in our decomposition $W(\rho_{M^0}) = \Wzero \ltimes \Waff$,
we saw in
Proposition \ref{propositionWheartisOmegainSL8example}
that $W(\rho_{M^0}) = \Wzero$,
and so
the affine reflection group $\Waff$ is trivial.
However, 
we have also seen in
Corollary \ref{cor:W-as-lattice-with-action} that
$W(\rho_{M^0})$, and thus $\Wzero$,
is a semi-direct product of
$M^0(F)/K_{M^0}$
and a group of order two.
Suppose we call these last two factors $\widetilde{W}^1$ and $C$, respectively.
Goldberg and Roche (\cite[\S11.8]{GoldbergRoche-Hecke})
show that
for some $w \in \widetilde{W}^1$ and $c \in C$,
the values of our cocycle $\muT(w,c)$ cannot be made non-trivial,
so our Hecke algebra cannot be written in the form
$\bC[C,\delta]\ltimes\cH(\widetilde{W}^1,q)$
for any $2$-cocycle $\delta$ on $C$
and any parameter function $q$.
However, this does not contradict any of our results.

In order to present our third example of an apparent
conflict with earlier results,
let us switch our attention to the case where $K^0 = G^0(F)_{x_0}$,
and so the type $(K^0,\rho^0)$ picks out a single
Bernstein block for $G^0(F)$.
Let $\fsz \in \IEC(G^0)$ denote the inertial equivalence
class associated to this block.
Let $\Pi_{\fsz}$ denote the Bernstein progenerator
for this block.
Thus, the category $\Rep^{\fsz}(G^0(F))$
is equivalent to the categories of unital right modules
over both our Hecke algebra
$\cH(G^0(F), (K^0, \rho^0))$
and over $\End_{G^0(F)}(\Pi_{\fsz})$,
so these algebras are a priori Morita equivalent.
One knows from 
\cite{MR2115455, OHARA2024183}
that they are in fact isomorphic.

Since $G^0$ is quasi-split and $M^0$ is a minimal Levi subgroup of $G^0$,
$\Rep^{\fsz}(G^0(F))$ is a principal-series block,
so the results of Solleveld \cite{Solleveld:principal} apply.
In particular, Theorem~2.7 of \cite{Solleveld:principal} describes the endomorphism algebra
$\End_{G^0(F)}(\Pi_\fsz)$, and hence our Hecke algebra, as follows:
\[
\cH(G^0(F), (K^0, \rho^0))
\simeq \End_{G^0(F)}(\Pi_\fsz)
\simeq \bC[\Gamma_\fsz] \widetilde\ltimes \cH(\fsz)^\circ.
\]

Here, in Solleveld's notation, $\cH(\fsz)^\circ$ is an affine Hecke algebra
associated to a root datum $\mathcal{R}_{\mathfrak{s}^0}$,
and $\Gamma_\fsz$ is a finite group acting on this root datum.
In the present setting, we have
$\cH(\fsz)^\circ = \bC[M^0(F)/K_{M^0}]$
and $\Gamma_\fsz = \{1,s\}$,
with $s$ as in Corollary~\ref{cor:W-as-lattice-with-action}.
We use the symbol $\widetilde\ltimes$ (where Solleveld uses $\ltimes$) to indicate that
in his presentation,
the action of $\Gamma_\fsz$ on $\cH(\fsz)^\circ$
need not coincide with the action induced by
its action on the root datum $\mathcal{R}_{\mathfrak{s}^0}$.
In the present example, this action incorporates nontrivial scalar factors,
as in \cite[\S11.8]{GoldbergRoche-Hecke},
discussed earlier.

Solleveld's description (of a general principal series block)
involves no explicit $2$-cocycle
and therefore appears to contradict
Corollary~\ref{corollaryanexampleofheckealgebrawithnon-trivialtwococycle}. But this is just an artefact of Solleveld encoding our non-trivial 2-cocycle in his action defining the semi-direct product.

\bibliographystyle{amsalpha}
\bibliography{ourbib}
\end{document}